\documentclass[11pt,a4paper,]{amsart}
\usepackage{amsfonts}
\usepackage{graphicx}
\usepackage{graphics}
\usepackage{epsfig}
\usepackage{anysize}
\usepackage{lipsum}
\usepackage{wrapfig}
\usepackage[shortlabels]{enumitem}
\usepackage{tikz-cd,caption,float}
\usepackage{stmaryrd}
\def\Int{\operatorname{Int}}
\def\Cl{\operatorname{Cl}}

\def\sfxy{$f:(X, \T, E)\to (Y, \S, E')\ $}

\marginsize{3cm}{3cm}{3cm}{3cm}
\newtheorem{theorem}{Theorem}[section]
\newtheorem{proposition}{Proposition}[section]
\newtheorem{lemma}{Lemma}[section]
\newtheorem{remark}{Remark}[section]
\newtheorem{corollary}{Corollary}[section]
\newtheorem{definition}{Definition}[section]
\newtheorem{example}{Example}[section]

\begin{document}
\setcounter{page}{1}
\title[Soft Somewhat Continuous and Soft Somewhat Open Functions]{\textit{Accepted in TWMS J. App. and Eng. Math. 2021}\\
	\vspace{1cm}Soft Somewhat Continuous and Soft Somewhat Open Functions}
\author{Zanyar A. Ameen}
\address{Deptartment of Mathematics, College of Science, University of Duhok, Duhok-42001, IRAQ}
\email{zanyar@uod.ac}

\author{Baravan A. Asaad}
\address{Deptartment of Mathematics, Faculty of Science, University of Zakho, Duhok-42002, IRAQ}
\address{Department of Computer Science, College of Science, Cihan University-Duhok, IRAQ}
\email{baravan.asaad@uoz.edu.krd}	

\author{Tareq M. Al-shami}
\address{Department of Mathematics, Sana’a University, P.O.Box 1247 Sana’a, Yemen}
\email{tareqalshami83@gmail.com}
	\begin{abstract}
		In this paper, we define a soft somewhat open set using the soft interior operator. We study main properties the class of soft somewhat open sets that is contained in the class soft somewhere dense sets. Then, we introduce the classes of soft somewhat continuous and soft somewhat open functions and soft somewhat homeomorphisms. Moreover, we study properties and characterizations of soft somewhat continuous and soft somewhat open functions. At last, we discuss topological invariants for soft somewhat homeomorphisms. Multiple examples are offered to clarify some invalid results.
		
		
		
	\end{abstract}
	\maketitle 
	\bigskip
	\def\P{\mathcal{P}}
	\def\T{\mathcal{T}}
	\def\S{\mathcal{S}}
	%
	\section{Introduction} \label{s1}

	In 1999, Molodtsov \cite{mold1999} suggested a different approach for dealing with problems of incomplete information under the name of soft set theory. This notion has been utilized in many directions, like: smoothness of function, Riemann integration, theory
	of measurement, probability theory, game theory and so on. The core concept of the theory of soft set is the nature of sets of parameters that provides a general framework for modeling uncertain data. This essentially contributes to the development of soft set theory during a short period of time. Maji et al. \cite{maji2003onsoft} studied  a (detailed) theoretical structure of soft set theory. In particular, they established some  operators and operations between soft sets. Then, some mathematicians reformulated the operators and operations between soft sets given in Maji et al.'s work as well as proposed different types of them; to see the recent contributions concerning soft operators and operations, we refer the reader to \cite{Al-shami}.
	
	In 2011, the concept of soft (general) topology was defined by Shabir and Naz \cite{shabir2011onsoft} and \c{C}a\u{g}man et al. \cite{cagman2011soft} independently. In 2013, Nazmul and Samanta \cite{softcont} defined soft continuity of functions. Then various generalizations of soft continuity and soft openness of functions appeared in the literature. For instance, soft $\alpha$-continuous functions \cite{akdag2014alpha}, soft semicontinuous functions \cite{softsemicont}, soft $\beta$-continuous functions \cite{yumak2015beta}, soft somewhere dense continuous \cite{shami2020sd-cont}, soft $\alpha$-open functions \cite{akdag2014alpha}, soft semi-open functions \cite{softsemicont}, soft $\beta$-open functions \cite{yumak2015beta}, soft somewhere dense open \cite{shami2020sd-cont}, and so on. Different kinds of belong and nonbelong relations were studied in \cite{shabir2011onsoft, El-shafei2018}. These relations led to the variety and abundance of the forms of the concepts and notions on soft topology.
	
	After this brief introduction, we recollect some preliminaries concepts in Section \ref{se2}. Then, we devote Section \ref{se3} to introduce the concept of soft somewhat open sets and study its relationships with some generalizations of soft open sets. The goals of Section \ref{se4} and Section \ref{se5} are to investigate soft somewhat continuous functions and soft somewhat open functions which are respectively weaker than soft semicontinuous and soft semi-open functions but stronger than soft somewhere dense continuous and soft somewhere dense open functions. In Section \eqref{se7}, we make a conclusion and propose some further works.

	\section{Preliminaries}\label{se2}\
	
	This section presents some basic definitions and notations that will be used in the sequel. Henceforth, we mean by $X$ an initial universe, $E$ a set of parameters and $\P(X)$ the power set of $X$.

	\begin{definition}\cite{mold1999}
		A pair $(F,E)=\{(e,F(e)):e\in E\}$ is said to be a soft set over $X$, where $F:E\to\P(X)$ is a (crisp) map. We write $F_E$ in place of the soft set $(F,E)$.
		
		The class of all soft sets on $X$ is symbolized by $SS_E(X)$ $($or simply $SS(X))$. If $A\subseteq E$, then it will be symbolized by $SS_A(X)$.
	\end{definition}
	
	\begin{definition}\cite{ramz,softcont}
		A soft set $F_E$ over $X$ is called:
		\begin{enumerate}[(i)]
			\item a soft element if $F(e) = \{x\}$ for all $e\in E$, where $x\in X$. It is denoted by $\{x\}_E$ $($or shortly $x)$.
			\item a soft point if there are $e\in E$ and $x\in X$ such that $F(e) = \{x\}$ and $F(e')= \emptyset$ for each $e'\neq e$. It is denoted by $P^x_e$. An expression $P^x_e\in F_E$ means that $x\in F(e)$.
		\end{enumerate}
	\end{definition}

	\begin{definition}\cite{ali2009onsome}
		The complement of $F_E$ is a soft set $X_E\setminus F_E$ $($or simply $F_E^c)$, where $F^c:E\to\P(X)$ is given by $F^c(e) = X\setminus F(e)$ for all $e\in E$.
	\end{definition}
	
	\begin{definition}\cite{mold1999}
		A soft subset $F_E$ over $X$ is called
		\begin{enumerate}[(i)]
			\item null if $F(e)=\emptyset$ for any $e\in E$.
			\item absolute if $F(e) = X$ for any $e\in E$.
		\end{enumerate}
		The null and absolute soft sets are respectively symbolized by $\Phi_E$ and $X_E$.\\
		Clearly, $X^c_E=\Phi_E$ and $\Phi_E^c=X_E$.
	\end{definition}
	
	\begin{definition}\cite{maji2003onsoft}
		Let $A, B\subseteq E$. It is said that $G_A$ is a soft subset of $H_B$ $($written by $G_A\sqsubseteq H_B)$ if $A\subseteq B$ and $F(e)\subseteq G(e)$ for any $e\in A$. We call $G_A$ soft equals to $H_B$ if $G_A\sqsubseteq H_B$ and $H_B\sqsubseteq G_A$.
	\end{definition}
	
	The definitions of soft union and soft intersection of two soft sets with respect to arbitrary subsets of $E$ was given by Maji et al. \cite{maji2003onsoft}. But it turns out that these definitions are misleading and ambiguous as reported by Ali et al. \cite{ali2009onsome}. Therefore, we follow the definitions given by Ali et al. \cite{ali2009onsome} and M. Terepeta \cite{terepeta}.
	
	\begin{definition}
		Let $\{F^\alpha_E:\alpha\in\Lambda\}$ be a collection of soft sets over $X$, where $\Lambda$ is any indexed set.
		\begin{enumerate}
			\item The intersection of $F^\alpha_E$, for $\alpha\in\Lambda$, is a soft set $G_E$ such that $G(e)=\bigcap_{\alpha\in\Lambda}F^\alpha(e)$ for each $e\in E$ and denoted by $G_E=\bigsqcap_{\alpha\in\Lambda}F^\alpha_E$.
			\item The union of $F^\alpha_E$, for $\alpha\in\Lambda$, is a soft set $G_E$ such that $G(e)=\bigcup_{\alpha\in\Lambda}F^\alpha(e)$ for each $e\in E$ and denoted by $G_E=\bigsqcup_{\alpha\in\Lambda}F^\alpha_E$.
		\end{enumerate}
	\end{definition}

	\begin{definition}\cite{shabir2011onsoft}
		A subfamily $\mathcal{T}$ of $SS_E(X)$ is called a soft topology on $X$ if
		\begin{enumerate}[(c1)]
			\item $\Phi_E$ and $X_E$ belong to $\mathcal{T}$,
			\item finite intersection of sets from $\mathcal{T}$ belongs to $\mathcal{T}$, and
			\item any union of sets from $\mathcal{T}$ belongs to $\mathcal{T}$.
		\end{enumerate}
		Terminologically, we call $(X, \T, E)$ a soft topological space on $X$. The elements of  $\T$ are called soft open sets, and their complements are called soft closed sets.
	\end{definition}
	
	Henceforward, $(X, \T, E)$ means a soft topological space.
	
	\begin{definition}\cite{shabir2011onsoft}
		Let $Y_E$ be a non-null soft subset of $(X, \T, E)$. Then $\T_Y:=\{G_E\bigsqcap Y_E:G_E\in\T\}$ is called a soft relative topology on $Y$ and $(Y, \T_Y, E)$ is a soft subspace of $(X, \T, E)$.
	\end{definition}
	
	\begin{definition}\cite{shabir2011onsoft}
		Let $F_E$ be a soft subset of $(X, \T, E)$. The soft interior of $F_E$ is the largest soft open set contained in $F_E$ and denoted by $\Int_X(F_E)$ $($or shortly $\Int(F_E))$. The soft closure of $F_E$ is the smallest soft closed set which contains $F_E$ and denoted by $\Cl_X(F_E)$ $($or simply $\Cl(F_E))$.
	\end{definition}
	
	\begin{lemma}\cite{hussein2011some}
		For a soft subset $G_E$ of $(X, \T, E)$, $\Int(G^c_E)=(\Cl(G_E))^c$ and $\Cl(G^c_E)=(\Int(G_E))^c$.
	\end{lemma}
	
	\begin{definition}
		A soft subset $G_E$ of $(X, \T, E)$ is called
		\begin{enumerate}[(i)]
			\item soft dense if $\Cl(G_E)=X_E$,
			\item soft co-dense if $\Int(G_E)=\Phi_E$
			\item soft semiopen \cite{chen2013semi} if $G_E\sqsubseteq\Cl(\Int(G_E))$,
			\item soft $\beta$-open \cite{yumak2015beta} if $G_E\sqsubseteq\Cl(\Int(\Cl(G_E)))$,
			\item soft somewhere dense \cite{shami2018somewhere} if $\Int(\Cl(G_E))\ne\Phi_E$ $($For a better connection between these soft sets, we force $\Phi_E$ to be soft somewhere dense$)$.
		\end{enumerate}
	\end{definition}
	
	We call $F_E$ a countable soft set if $F(e)$ is countable for each $e\in E$.

	\begin{definition}
		A soft topological space $(X, \T, E)$ is called
		\begin{enumerate}[(i)]
			\item soft separable \cite{rong-separable} if it has a countable soft dense subset.
			\item soft hyperconnected \cite{kandil2014hyperconnected} if any pair of non-null soft open subsets intersect.
			\item soft connected \cite{connected} if it cannot be written as a union of two disjoint soft open sets.
			\item soft compact \cite{compact} if every cover of $X$ by soft open sets has a finite subcover. It is soft locally compact if each soft point has a soft compact neighborhood.
			\item soft metrizable \cite{metric} if $\T$ is induced by soft metric space.
		\end{enumerate}
	\end{definition}
	
	\begin{definition}\label{T_i}\cite{shabir2011onsoft,bayramov}
		A soft topological space $(X, \T, E)$ is called
		\begin{enumerate}[(i)]
			\item soft $T_0$ if for each $P^x_e,P^y_e\in X$ with $P^x_e\ne P^y_e$, there exist soft open sets $G_E, H_E$ such that $P^x_e\in G_E$, $P^y_e\notin G_E$ or $P^y_e\in H_E$, $P^x_e\notin H_E$.
			\item soft $T_1$ if for each $P^x_e,P^y_e\in X$ with $P^x_e\ne P^y_e$, there exist soft open sets $G_E, H_E$ such that $P^x_e\in G_E$, $P^y_e\notin G_E$ and $P^y_e\in H_E$, $P^x_e\notin H_E$,
			\item soft $T_2$ $($soft Hausdorff$)$ if for each $P^x_e,P^y_e\in X$ with $P^x_e\ne P^y_e$, there exist soft open sets $G_E, H_E$ containing $P^x_e, P^y_e$ respectively such that $G_E\bigsqcap H_E=\Phi_E$. 
		\end{enumerate}
	\end{definition}

	\begin{definition}
		Let $(X, \T, E)$ and $(Y, \S, E')$ be soft topological spaces. A soft function $f:(X, \T, E)\to(Y, \S, E')$ is called
		\begin{enumerate}[(i)]
			\item soft continuous \cite{softcont} (resp., soft semicontinuous \cite{softsemicont}, soft $SD$-continuous \cite{shami2020sd-cont}, soft $\beta$-continuous \cite{yumak2015beta}) if the inverse image of each soft open subset of $(Y, \S, E')$ is a soft open (resp., soft semiopen, soft somewhere dense, $\beta$-open) subset of $(X, \T, E)$.
			\item soft open \cite{softcont} (resp., soft semiopen \cite{softsemicont}, soft $SD$-open \cite{shami2020sd-cont}, soft $\beta$-open \cite{yumak2015beta}) if the image of each soft open subset of $(X, \T, E)$ is a soft open (resp., soft semiopen, soft somewhere dense, $\beta$-open) subset of $(Y, \S, E')$.
			\item soft homeomorphism \cite{softcont} if it is one to one soft open and soft continuous from $(X, \T, E)$ onto $(Y, \S, E')$.
		\end{enumerate}
	\end{definition}
	
	For the definition of soft functions between collections of all soft sets, we refer the reader to \cite{kharal2011softmapping}. Henceforward, by the word "function" we mean "soft function".

	\section{Soft Somewhat Open Sets}\label{se3}
	In this section, we introduce the concept of soft somewhat open sets and establish main properties. With the help of examples, we show the relationships between soft somewhat open sets and some generalizations of soft open sets such that soft semiopen and soft somewhere dense sets.
	\begin{definition}\label{defn1}
		A subset $G_E$ of a soft topological space $(X, \T, E)$ is said to be soft somewhat open $($briefly soft $sw$-open$)$ if either $G_E$ is null or $\ \Int(G_E)\neq\Phi_E$.
	\end{definition}
	
	The complement of each soft $sw$-open set is called soft $sw$-closed. That is, a set $F_E$ is soft $sw$-closed if $\Cl(F_E)\neq X_E$ or $F_E=X_E$.
	
	\begin{remark}\label{defn-sw}
		Let $(X, \T, E)$ be a soft topological space.
		\begin{enumerate}[(a)]
			\item A non-null set $G_E$ over $X$ is soft $sw$-open iff there is a soft open set $U_E$ such that $\Phi_E\ne U_E\sqsubseteq G_E$.
			\item A proper set $H_E$ over $X$ is soft $sw$-closed iff there is a soft closed set $F_E$ such that $H_E\sqsubseteq F_E\ne X_E$.
		\end{enumerate}
	\end{remark}
	
	\begin{proposition}\label{}
		\begin{enumerate}[(a)]
			\item Every superset of a soft $sw$-open set is soft $sw$-open.
			\item Every subset of a soft $sw$-closed set is soft $sw$-closed.
		\end{enumerate}
	\end{proposition}
	\begin{proof}
		Straightforward.
	\end{proof}
	
	\begin{proposition}
		A non-null soft set is soft $sw$-open iff it is a soft neighbourhood of a soft point.
	\end{proposition}
	\begin{proof}
		Let $G_E$ be a non-null soft $sw$-open set. Then there is a soft open set $U_E$ such that $\Phi_E\ne U_E\sqsubseteq G_E$. Therefore, $G_E$ is a soft neighbourhood of all soft points in $U_E$. Conversely, let $G_E$ be a soft neighbourhood of a soft point $P^x_e$. Then there is a soft open set $U_E$ such that $P^x_e\in U_E\sqsubseteq G_E$. Hence, we obtain $\Int(G_E)\neq \Phi_E$, as required.
	\end{proof}
	
	\begin{proposition}\label{union-sw}
		Any union of soft $sw$-open sets is soft $sw$-open.
	\end{proposition}
	\begin{proof}
		Let $\{G_E^{\alpha}:\alpha\in\Lambda\}$ be any collection of soft $sw$-open subsets of a soft topological space $(X, \T, E)$. Now
		\begin{eqnarray*}
			\Int(\bigsqcup_{\alpha\in\Lambda} G_E^\alpha)&\sqsupseteq&\bigsqcup_{\alpha\in\Lambda}\Int(G_E^\alpha))\neq\Phi_E.
		\end{eqnarray*}
		Thus $\bigsqcup_{\alpha\in\Lambda} G_E^\alpha$ is soft $sw$-open.
	\end{proof}
	
	\begin{corollary}
		Any intersection of soft $sw$-closed sets is soft $sw$-closed.
	\end{corollary}
	
	The intersection of two soft $sw$-open sets need not be soft $sw$-open, as showing in the next example:
	\begin{example}\label{intersection}
		Let $\mathbb{R}$ be the set of real numbers and $E=\{e_1, e_2\}$ be a set of parameters. Let $\T$ be the soft topology on $\mathbb{R}$ generated by  $\{(e_i, B(e_i)):B(e_i)=(a_i, b_i); a_i, b_i\in\mathbb{R}; a_i\leq b_i; i=1,2\}$. Take soft $sw$-open sets $G_E=\{(e_1, [0,1]), (e_2, [0,1])\}$ and $H_E=\{(e_1, [1,2]), (e_2, [1,2])\}$ over $\mathbb{R}$, then $G_E\bigsqcap H_E\ne\Phi_E$ but $\Int(G_E\bigsqcap H_E)=\Phi_E$.	
	\end{example}

	\begin{remark}
		The intersection of a soft $sw$-open set with another soft open, soft closed or soft dense set need not be a soft $sw$-open set, and counterexamples showing this are easy to find.
	\end{remark}
	
	The result below explains the conditions under which the intersection of soft $sw$-open and soft open sets is a soft $sw$-open set.
	\begin{proposition}\label{}
		The intersection of two soft $sw$-open sets in a soft hyperconnected space $(X, \T, E)$ is a soft $sw$-open set.
	\end{proposition}
	\begin{proof}
		If one of the two soft $sw$-open sets is null, the proof is trivial. 
		Suppose $G_E$ and $H_E$ are two soft $sw$-open sets. Then $\Int(G_E)=U_E\neq\Phi_E$ and $\Int(H_E)=V_E\neq\Phi_E$. Now, $\Int(G_E\bigsqcap H_E)=\Int(G_E)\bigsqcap \Int(H_E)=U_E\bigsqcap V_E$. Since $(X, \T, E)$ is soft hyperconnected, $U_E\bigsqcap V_E\neq \Phi_E$. Thus $\Int(G_E\bigsqcap H_E)\neq \Phi_E$; hence, we obtain the desired result.
	\end{proof}
	
	\begin{corollary}
		The intersection of soft $sw$-open and soft open sets in a soft hyperconnected space $(X, \T, E)$ is a soft $sw$-open set.
	\end{corollary}
	\begin{corollary}
		The family of soft $sw$-open subsets of a soft hyperconnected space $(X, \T, E)$ forms a soft topology.
	\end{corollary}
	
	\begin{lemma}\label{intersectionwithdense}
		Let $G_E, D_E$ be subsets of $(X, \T, E)$. If $G_E$ is $sw$-open and $D_E$ is soft dense over $X$, then $G_E\bigsqcap D_E$ is soft $sw$-open over $D$.
	\end{lemma}
	\begin{proof}
		Since $\Int_D(G_E\bigsqcap D_E)=\Int_D(G_E)\bigsqcap D_E\sqsupseteq \Int(G_E)\bigsqcap D_E\ne\Phi_E$ (as $D_E$ is soft dense), so $G_E\bigsqcap D_E$ is soft $sw$-open over $D$.
	\end{proof}
	
	\begin{lemma}\label{sw(Y)>>sw(X)}
		Let $(Y, \T_Y, E)$ be a soft open subspace of $(X, \T, E)$ and let $G_E\sqsubseteq Y_E$. Then $G_E$ is soft $sw$-open over $Y$ iff it is soft $sw$-open over $X$.
	\end{lemma}
	\begin{proof}
		Assume $G_E$ is soft $sw$-open over $Y$. There exists a soft open set $U_E$ over $Y$ such that $\Phi_E\ne U_E\sqsubseteq G_E$. Since $Y_E$ is soft open over $X$, so $U_E$ is soft open over $X$. Hence $G_E$ is soft $sw$-open over $X$.
		
		Conversely, assume $G_E$ is soft $sw$-open over $X$. That is $\Int_X(G_E)\ne\Phi_E$. By Theorem 2 in \cite{shabir2011onsoft}, $\Int_X(G_E)\sqsubseteq\Int_Y(G_E)$, therefore $G_E$ is soft $sw$-open over $Y$.
	\end{proof}

	The following example shows that the above result is not true if $Y_E$ is soft dense in $\widetilde{X}$.
	\begin{example}
		Let $X=\{w,x,y,z\}$ and $E=\{e_1,e_2\}$. Set $\T=\{\Phi_E, F_E, G_E, H_E, X_E\}$, where
		\begin{eqnarray*}
			F_E&=&\{(e_1,\{x,z\}), (e_2, \{w,x\})\}\\
			G_E&=&\{(e_1,X), (e_2, \{y,z\})\}\\
			H_E&=&\{(e_1,\{x,z\}), (e_2, \emptyset)\}.
		\end{eqnarray*}
		Take $Y=\{x,y\}$, so $\T_Y=\{\Phi_E, I_E, J_E, K_E, Y_E\}$, where
		\begin{eqnarray*}
			I_E&=&\{(e_1,\{x\}), (e_2, \{x\})\}\\
			J_E&=&\{(e_1,Y), (e_2, \{y\})\}\\
			K_E&=&\{(e_1,\{x\}), (e_2, \emptyset)\}\\
			Y_E&=&\{(e_1,\{x,y\}), (e_2,\{x,y\})\}.
		\end{eqnarray*}
		The set $I_E$ is soft $sw$-open over the soft dense set $Y$ but not soft $sw$-open over $X$.
	\end{example}

	\begin{lemma}\label{cl=clint semi}
		Let $G_E$ be a subset of $(X, \T, E)$. Then $G_E$ is soft semiopen iff $\Cl(G_E)=\Cl(\Int(G_E))$.
	\end{lemma}
	\begin{proof}
		If $G_E$ is soft semiopen, then $G_E\sqsubseteq\Cl(\Int(G_E))$ and so $\Cl(G_E)\sqsubseteq\Cl(\Int(G_E))$. For other side of inclusion, we always have $\Int(G_E)\sqsubseteq G_E$. Therefore $\Cl(\Int(G_E))\sqsubseteq \Cl(G_E)$. Thus $\Cl(G_E)=\Cl(\Int(G_E))$.
		
		Conversely, assume that $\Cl(G_E)=\Cl(\Int(G_E))$, but $G_E\sqsubseteq\Cl(G_E)$ always, so $G_E\sqsubseteq\Cl(\Int(G_E))$. Hence $G_E$ is soft semiopen.
	\end{proof}

	\begin{lemma}\label{semi=int not 0}
		Let $G_E$ be a non-null subset of $(X, \T, E)$. If $G_E$ is soft semiopen, then $\Int(G_E)\neq\Phi_E$.
	\end{lemma}
	\begin{proof}
		Suppose otherwise that if $G_E$ is a non-null soft semiopen set such that $\Int(G_E)=\Phi_E$, by Lemma \ref{cl=clint semi}, $\Cl(G_E)=\Phi_E$ which implies that $G_E=\Phi_E$. Contradiction!
	\end{proof}
	
	\begin{remark}\label{w-relation sets}
		Since $\Int(G_E)\sqsubseteq\Int(\Cl(G_E))$ for each soft set $G_E$ in a soft topological space $(X, \T, E)$, so each soft $sw$-open set is soft somewhere dense.
	\end{remark}

	Next, we put Remark~\ref{w-relation sets}, Lemma~\ref{semi=int not 0} and Proposition 2.8 in  \cite{shami2018somewhere} into the following diagram:
	
	\begin{figure}[H]
		\centering
		\begin{tikzcd}[arrows=rightarrow]
			&\text{soft semiopen set}\arrow[dd]\arrow[rr]&& \text{soft $\beta$-open set}\arrow[dd]\\
			&&&\\
			&\text{soft $sw$-open set}\arrow[rr]&&\text{soft somewhere dense set}
		\end{tikzcd}
		\caption*{\quad Diagram I: Relationship between some generalizations of soft open sets}\label{D1}
	\end{figure}
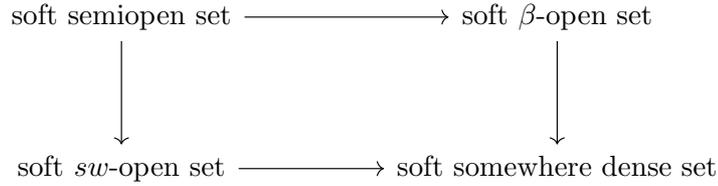
	
	In general, none of these implications can be replaced by equivalence as shown below:
	\begin{example}
		Consider the soft topology defined in Example \ref{intersection}. The soft set of rational numbers $\mathbb{Q}_E$ over $\mathbb{R}$ is soft $\beta$-open (consequently, is soft somewhere dense) but not soft $sw$-open (consequently, is not soft semi-open). On the other hand, the set
		$\{(e_1, (0,1)),(e_2, \{2\})\}$ is clearly soft $sw$-open but not soft semiopen. The soft set $F_E$ given in Example 2.9 in \cite{shami2018somewhere} is soft somewhere dense but not soft $\beta$-open.
	\end{example}
	
	\begin{lemma}\label{closure}\cite[Lemma 2.24]{shami2018somewhere}
		Let $G_E$ be a subset of $(X, \T, E)$. Then $\Cl(G_E)\bigsqcap U_E\sqsubseteq\Cl(G_E\bigsqcap U_E)$ for each soft open set $U_E$ over $X$.
	\end{lemma}
	
	\begin{lemma}\label{open X semi}
		Let $G_E, H_E$ be subsets of $(X, \T, E)$. If $G_E$ is soft open and $H_E$ is soft semiopen, then $G_E\bigsqcap H_E$ is soft semiopen over $X$.
	\end{lemma}
	\begin{proof}
		Assume $H_E$ is soft semiopen and $G_E$ is soft open. By Theorem 3.1 in \cite{chen2013semi}, there exists a soft open set $U_E$ over $X$ such that $U_E\sqsubseteq H_E\sqsubseteq\Cl(U_E)$. Now $U_E\bigsqcap G_E\sqsubseteq H_E\bigsqcap G_E\sqsubseteq\Cl(U_E)\bigsqcap G_E$. By Lemma \ref{closure}, $U_E\bigsqcap G_E\sqsubseteq H_E\bigsqcap G_E\sqsubseteq\Cl(U_E\bigsqcap G_E)$ and since $U_E\bigsqcap G_E$ is soft open, therefore by Theorem 3.1 in \cite{chen2013semi},  $H_E\bigsqcap G_E$ is soft semiopen over $X$.
	\end{proof}
	
	\begin{lemma}\label{open X semi in open}
		Let $G_E, H_E$ be subsets of $(X, \T, E)$. If $G_E$ is soft open and $H_E$ is soft semiopen, then $G_E\bigsqcap H_E$ is soft semiopen over $G$.
	\end{lemma}
	\begin{proof}
		Apply the same steps in the proof of above lemma and use the statement that $\Cl(U_E)\bigsqcap G_E=\Cl_{G_E}(U_E)$.
	\end{proof}
	
	\begin{lemma}\label{semi = sw X open}
		A subset $G_E$ of $(X, \T, E)$ is soft semiopen iff $G_E\bigsqcap U_E$ is soft $sw$-open for each soft open set $U_E$ over $X$.
	\end{lemma}
	\begin{proof}
		Since each soft semiopen set is soft $sw$-open and by Lemma \ref{open X semi}, the intersection of a soft semiopen set with a soft open set is semiopen, so the first part follows.
		
		Conversely, let $P^x_e\in G_E$ and assume that $G_E\bigsqcap U_E$ is soft $sw$-open for each soft open set $U_E$ over $X$. That is $\Int(G_E\bigsqcap U_E)\neq\Phi_E$. But  $\Phi_E\neq\Int(G_E\bigsqcap U_E)=\Int(G_E)\bigsqcap \Int(U_E)=\Int(G_E)\bigsqcap U_E$, which implies that $P^x_e\in \Cl(\Int(G_E))$ and so $G_E\sqsubseteq\Cl(\Int(G_E))$. This proves that $G_E$ is soft semiopen.
	\end{proof}
	
	\begin{lemma}
		Let $F_E$ be a subset of $(X, \T, E)$. If $F_E$ is soft semiclosed and soft somewhere dense, it is soft $sw$-open.
	\end{lemma}
	\begin{proof}
		Directly follows from Lemma \ref{cl=clint semi} which implies that $F_E$ is semiclosed iff $\Int(\Cl(F_E))=\Int(F_E)$.
	\end{proof}

	\section{Soft Somewhat Continuous Functions}\label{se4}
	We devote this section to presenting the concepts of soft somewhat continuous functions (briefly soft $sw$-continuous) and giving several characterizations of it. In addition, we illustrate its relationships with some types of soft continuity. Finally, we derive some results related to soft separable and hyperconnected spaces.
	
	\begin{definition}
		Let $(X, \T, E)$ and $(Y, \S, E')$ be soft topological spaces. A function \sfxy is said to be soft $sw$-continuous if the inverse image of each soft open set over $Y$ is soft $sw$-open over $X$.
	\end{definition}
	
	The above definition can be stated as:
	\begin{remark}
		A function \sfxy is soft $sw$-continuous if for each $P^x_e\in X$ and each soft open set $V_{E'}$ over $Y$ containing $f(P^x_e)$, there exists a soft  $sw$-open set $U_E$ over $X$ containing $P^x_e$ such that $f(U_E)\sqsubseteq V_{E'}$.
	\end{remark}
	
	From Diagram I, we conclude that
	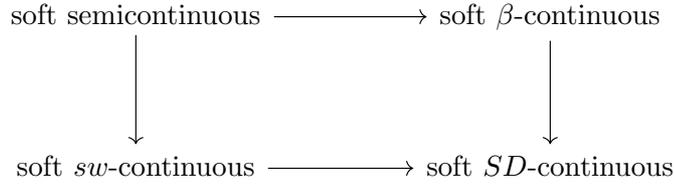
\begin{figure}[H]
		\centering
		\begin{tikzcd}[arrows=rightarrow]
			&\text{soft semicontinuous}\arrow[dd]\arrow[rr]&& \text{soft $\beta$-continuous}\arrow[dd]\\
			&&&\\
			&\text{soft $sw$-continuous}\arrow[rr]&&\text{soft $SD$-continuous}
		\end{tikzcd}
		\caption*{\quad Diagram II: Relationship between some generalizations of soft continuity}\label{D2}
	\end{figure}
	
	None of the implications in the above diagram is reversible.
	\begin{example}\label{sw-not semicont}
		Let $X=\{x,y,z\}$ and $E=\{e_1,e_2\}$. Put $\T=\{\Phi_E, F_E, G_E, X_E\}$, where $F_E=\{(e_1,\{y\})$, $(e_2,\{y\})\}$, $G_E=\{(e_1,\{x,z\}), (e_2,\{x,z\})\}\}$ and $\S=\{\Phi_E, H_E, X_E\}$, where $H_E=\{(e_1,X), (e_2,\{x,y\})\}$. Let $f:(X, \T, E)\to (X, \S, E)$ be the soft identity function. Then $f$ is soft $sw$-continuous but not soft semicontinuous.
	\end{example}
	
	\begin{example}\label{1-1+sw=not=semi}
		Let $X=\mathbb{R}$ be the set of real numbers and $E=\{e\}$ be a set of parameters. Let $\T$ be the soft topology on $\mathbb{R}$ generated by  $\{(e, B(e)):B(e)=(a, b); a, b\in\mathbb{R}; a<b\}$. Define a soft function $f:(X, \T, E)\to (X, \T, E)$ by
		\[
		f(x) = \begin{cases}
			x, & \text{if } x\notin\{0, 1\}_E;\\
			0, & \text{if } x=1;\\
			1, & \text{if } x=0.
		\end{cases}
		\]
		One can easily show $f$ is soft $sw$-continuous (consequently, soft $SD$-continuous) because the inverse image of any soft basic open set always contains some soft basic open, so its soft interior cannot be null. On the other hand $f$ is not soft $\beta$-continuous. Take the soft open set $G_E=\{(e, (-\varepsilon, \varepsilon))\}$, where $\varepsilon<1$. Therefore $$f^{-1}(G_E)=\{(e, (-\varepsilon, 0))\}\bigsqcup\{(e, (0, \varepsilon))\}\bigsqcup\{(e, \{1\})\}.$$ But $\Cl(\Int(\Cl(f^{-1}(G_E))))=\{(e, [-\varepsilon, \varepsilon])\}$ and so  $f^{-1}(G_E)\not\sqsubseteq\Cl(\Int(\Cl(f^{-1}(G_E))))$. In conclusion,  $f$ cannot be soft $\beta$-continuous (consequently, is not soft semicontinuous).
	\end{example}
	
	\begin{example}
		Let $(X, \T, E)$  be the soft topological space given in Example \ref{1-1+sw=not=semi} and let $f:(X, \T, E)\to (X, \T, E)$ be defined by
		\[f(x) =
		\begin{cases}
			0, & x \not\in \mathbb{Q}_E;\\
			1, & x \in \mathbb{Q}_E. 	
		\end{cases}
		\]
		Then $f$ is soft $SD$-continuous but not soft $sw$-continuous. The inverse image of any soft open set containing only $1$ is $\mathbb{Q}_E$ which is not soft $sw$-open over $X$.
	\end{example}
	
	
	\begin{definition}
		For a subset $G_E$ of a soft topological space $(X, \T, E)$, we introduce the following:
		\begin{enumerate}[(i)]
			\item $\Cl_{sw}(G_E)=\bigcap\{F_E: F_E \text{ is soft $sw$-closed over $X$ and } G_E\sqsubseteq F_E\}$.
			\item $\Int_{sw}(G_E)=\bigcup\{O_E: O_E \text{ is soft $sw$-open over $X$ and } O_E\sqsubseteq G_E\}$.
		\end{enumerate}
	\end{definition}
	
	\begin{proposition}
		Let $(X, \T, E)$ and $(Y, \S, E')$ be soft topological spaces. For a function \sfxy, the following  are equivalent:
		\begin{enumerate}[(1)]
			\item $f$ is soft $sw$-continuous,
			\item $f^{-1}(F_{E'})$ is soft $sw$-closed set over $X$, for each soft closed set $F_{E'}$ over $Y$,
			\item $f(\Cl_{sw}(G_E))\sqsubseteq \Cl(f(G_E))$, for each set  $G_E$ over $X$,
			\item $\Cl_{sw}(f^{-1}(H_{E'}))\sqsubseteq f^{-1}(Cl(H_{E'}))$, for each set  $H_{E'}$ over $Y$,
			\item $f^{-1}(\Int(H_{E'}))\sqsubseteq \Int_{sw}(f^{-1}(H_{E'}))$, for each set  $H_{E'}$ over $Y$,
		\end{enumerate}
	\end{proposition}
	\begin{proof}
		Follows from the definition of soft $sw$-continuity.
	\end{proof}
	
	\begin{definition}\cite[Definition 3.10]{shami2020sd-cont}
		Let $(X, E)$ and $(Y, E')$ be soft sets and let $A_E\in(X, E)$. The restriction of $f:(X, E)\to(Y, E')$ is the soft function $f_{A_E}:(X, E)\to(Y, E')$ defined by $f_{A_E}(P^x_e)=f(P^x_e)$ for all $P^x_e\in A_E$. An extension of a soft function $f$ is a soft function $g$ such that $f$ is a restriction of $g$
	\end{definition}

	\begin{theorem}
		Let $(X, \T, E)$ and $(Y, \S, E')$ be soft topological spaces and let $D_E$ be a soft dense subspace over $X$. If \sfxy is soft $sw$-continuous over $X$, then $f|_{D_E}$ is soft $sw$-continuous over $D$.
	\end{theorem}
	\begin{proof}
		Standard (by using Lemma \ref{intersectionwithdense}).
	\end{proof}
	
	\begin{theorem}
		Let $(X, \T, E)$ and $(Y, \S, E')$ be soft topological spaces. Let \sfxy be a function and $\{G_E^\alpha: \alpha\in\Lambda\}$ be a soft open cover of $X$. Then $f$ is soft $sw$-continuous, if $f|_{G_E^\alpha}$ is soft $sw$-continuous for each $\alpha\in\Lambda$.
	\end{theorem}
	\begin{proof}
		Let $V_{E'}$ be a soft open set over $Y$. By assumption, $\left(f|_{G_E^\alpha}\right)^{-1}(V_{E'})$ is soft $sw$-open over $G_E^\alpha$. By Lemma \ref{sw(Y)>>sw(X)}, $\left(f|_{G_E^\alpha}\right)^{-1}(V_{E'})$ is soft $sw$-open over $X$ for each $\alpha\in\Lambda$.  But $$f^{-1}(V_{E'})=\bigsqcup_{\alpha\in\Lambda}\left[\left(f|_{G_E^\alpha}\right)^{-1}(V_{E'}) \right],$$ which a union of soft $sw$-open sets and by Lemma \ref{union-sw}, $f^{-1}(V_{E'})$ is soft $sw$-open over $X$. Hence $f$ is soft $sw$-continuous.
	\end{proof}

	\begin{theorem}
		Let $(X, \T, E)$ and $(Y, \S, E')$ be soft topological spaces and let $W_E$ be a soft open set over $X$. If $f:(W, \T_W, E)\to (Y, \S, E')$ is a soft $sw$-continuous function such that $f(W_E)$ is soft dense over $Y$, then each extension function of $f$ over $X$ is soft $sw$-continuous.
	\end{theorem}
	\begin{proof}
		Let $g$ be an extension of $f$ and let $V_{E'}$ be a (non-null) soft open set over $Y$. If $g^{-1}(V_{E'})=\Phi_E$, then $g$ is trivially soft $sw$-continuous. Suppose $g^{-1}(V_{E'})\ne\Phi_E$. By density of $f(W_E)$, $f(W_E)\bigsqcap V_{E'}\ne\Phi_{E'}$ which implies that $W_E\bigsqcap f^{-1}(V_{E'})\ne\Phi_E$. Therefore $f^{-1}(V_{E'})\ne\Phi_E$. By assumption, there exists a non-null soft open set $U_E$ over $W$ such that $$U_E=U_E\bigsqcap W_E\sqsubseteq f^{-1}(V_{E'})\bigsqcap W_E= g^{-1}(V_{E'})\bigsqcap W_E\sqsubseteq g^{-1}(V_{E'}).$$ By Lemma \ref{sw(Y)>>sw(X)}, $U_E$ is a soft open set over $X$ and so $\Phi_E\ne U_E\sqsubseteq g^{-1}(V_{E'})$. Thus $g$ is soft $sw$-continuous over $X$.
	\end{proof}
	
	
	\begin{theorem}\label{semi f=sw in O}
		Let $(X, \T, E)$ and $(Y, \S, E')$ be soft topological spaces. A function \sfxy is soft semicontinuous iff  $f|_{U_E}$ is $sw$-continuous for each soft open set  $U_E$ over $X$.
	\end{theorem}
	\begin{proof}
		Assume that $f$ is soft semicontinuous and $U_E$ is any soft open over $X$. Let $G_{E'}$ be a soft open set over $Y$. Then $f^{-1}(G_{E'})$ is soft semiopen and so, by Lemma \ref{open X semi in open}, $\left( f|_{U_E}\right)^{-1}(G_{E'})=f^{-1}(G_{E'})\bigsqcap U_E$ is soft semiopen over $U$. Thus $f|_{U_E}$ is soft semicontinuous and hence soft $sw$-continuous.
		
		Conversely, suppose that $f|_{U_E}$ is soft $sw$-continuous for each soft open set  $U_E$ over $X$. Let $H_{E'}$ be soft open over $Y$. Then $\left( f|_{U_E}\right)^{-1}(H_{E'})=f^{-1}(H_{E'})\bigsqcap U_E$ is soft $sw$-open over $U$. Since $U_E$ is a soft open set  over $X$, by Lemma \ref{sw(Y)>>sw(X)}, $f^{-1}(H_{E'})\bigsqcap U_E$ is soft $sw$-open over $X$ and so, by Lemma~\ref{semi = sw X open}, $f^{-1}(H_{E'})$ is soft semiopen over $X$. Thus $f$ is soft semicontinuous.
	\end{proof}

	\begin{theorem}\label{characterization1 sw-cont}
		Let $(X, \T, E)$ and $(Y, \S, E')$ be soft topological spaces. For a function \sfxy, the following  are equivalent:
		\begin{enumerate}[(1)]
			\item $f$ is soft $sw$-continuous,
			\item for each soft open set  $V_{E'}$ over $Y$ with $f^{-1}(V_{E'})\neq\Phi_E$, there exists a non-null soft open set $U_E$ over $X$ such that $U_E\sqsubseteq f^{-1}(V_{E'})$,
			\item for each soft closed set  $F_{E'}$ over $Y$ with $f^{-1}(F_{E'})\neq X_E$, there exists a proper soft closed $K_E$ over $X$ such that $f^{-1}(F_{E'})\sqsubseteq K_E$,
			\item for each soft dense set  $D_E$ over $X$, then $f(D_E)$ is soft dense over $f(X)$.
		\end{enumerate}
	\end{theorem}
	\begin{proof}
		(1)$\implies$(2) Remark \ref{defn-sw} and the definition of $sw$-continuity.
		
		(2)$\implies$(3) Let $F_{E'}$ be a soft closed set over $Y$ such that $f^{-1}(F_{E'})\neq X_E$. Then $Y_{E'}\setminus F_{E'}$ is soft open over $Y$ with $f^{-1}(Y_{E'}\setminus F_{E'})\neq\Phi_E$. By (2), there exists a soft open set $U_E$ over $X$ such that $\Phi_E\neq U_E\sqsubseteq f^{-1}(Y_{E'}\setminus F_{E'})=X_E\setminus f^{-1}(F_{E'})$. This implies that $f^{-1}(F_{E'})\sqsubseteq X_E\setminus U_E\neq X_E$. If $K_E=X_E\setminus U_E$, then $K_E$ is a proper soft closed set that satisfies the required property.
		
		(3)$\implies$(4) Let $D_E$ be soft dense over $X$. We need to prove that $f(D_E)$ is soft dense over $f(X)$. Suppose that $f(D_E)$ is not soft dense over $f(X)$. There exists a proper soft closed set $F_{E'}$ such that  $f(D_E)\sqsubseteq F_{E'}\sqsubset f(X_E)$. Therefore $D_E\sqsubseteq f^{-1}(F_{E'})$. By (3), there exists a soft closed set $K_E$ over $X$ such that $D_E\subseteq f^{-1}(F_{E'})\sqsubseteq K_E\ne X_E$. This contradicts that $D_E$ is soft dense over $X$. Thus (4) holds.
		
		(4)$\implies$(1) With out loss of generality, let $H_{E'}$ be a soft open set over $Y$ with $f^{-1}(H_{E'})\neq\Phi_E$, because if $f^{-1}(H_{E'})=\Phi_E$, then it is trivially soft $sw$-open. Suppose that $f^{-1}(H_{E'})$ is not soft $sw$-open. That is $\Int(f^{-1}(H_{E'}))=\Phi_E$. Therefore $\Cl(X_E\setminus f^{-1}(H_{E'})=X_E$. This implies that $X_E\setminus f^{-1}(H_{E'})$ is soft dense over $X$. By (4), $f(X_E\setminus f^{-1}(H_{E'}))$ is soft dense over $f(X)$, i.e., $\Cl(f(X_E\setminus f^{-1}(H_{E'})))=f(X_E)$. This yields that  $\Cl(f(X_E)\setminus H_{E'})=f(X_E)\setminus H_{E'}=f(X_E)$ and so $H_{E'}=\Phi_{E'}$. Contradiction to the choice of $H_{E'}$. It follows that $\Int(f^{-1}(H))$ must not be null. Thus $f^{-1}(H_{E'})$ is soft $sw$-open over $X$.
	\end{proof}

	\begin{corollary}\label{characterization2 sw-cont}
		Let $(X, \T, E)$ and $(Y, \S, E')$ be soft topological spaces. For a one to one function $f$ from a space $(X, \T, E)$ onto a space $(Y, \S, E')$, the following  are equivalent:
		\begin{enumerate}[(1)]
			\item $f$ is soft $sw$-continuous,
			\item for each soft co-dense set  $N_E$ over $X$, $f(N_E)$ is soft co-dense over $Y$.
		\end{enumerate}
	\end{corollary}
	
	We complete this section by discussing two results related to soft separable and hyperconnected spaces.
	
	\begin{theorem}\label{separable}
		Let $(X, \T, E)$ and $(Y, \S, E')$ be soft topological spaces, and let $f$ be a function from $(X, \T, E)$ onto $(Y, \S, E')$. If $f$ is soft $sw$-continuous and $(X, \T, E)$ is soft separable, then $(Y, \S, E')$ is soft separable.
	\end{theorem}
	\begin{proof}
		Let $D_E$ be a countable soft dense set over $X$. Clearly $f(D_E)$ is countable. By Theorem \ref{characterization1 sw-cont} (4), $f(D_E)$ is soft dense over $f(X)=Y$. Therefore $(Y, \S, E')$ is soft separable.
	\end{proof}
	
	\begin{theorem}\label{hyperconnected}
		Let $(X, \T, E)$ and $(Y, \S, E')$ be soft topological spaces. If $f$ is a soft $sw$-continuous from $(X, \T, E)$ onto $(Y, \S, E')$ and $(X, \T, E)$ is soft hyperconnected, then $(Y, \S, E')$ is soft hyperconnected.
	\end{theorem}
	\begin{proof}
		Let $G_{E'}, H_{E'}$ be any two soft open sets over $Y$ with $G_{E'}\ne\Phi_{E'}\ne H_{E'}$. Since $f$ is soft $sw$-continuous, then $\Int(f^{-1}(G_{E'}))\ne\Phi_E\ne\Int(f^{-1}(H_{E'}))$. But $(X, \T, E)$ is soft hyperconnected, so $$\Int(f^{-1}(G_{E'}))\bigsqcap\Int(f^{-1}(H_{E'}))\ne\Phi_E.$$ If $$x\in\Int(f^{-1}(G_{E'}))\bigsqcap\Int(f^{-1}(H_{E'}))\sqsubseteq f^{-1}(G_{E'})\bigsqcap f^{-1}(H_{E'}),$$ then $f(x)\in G_{E'}\bigsqcap H_{E'}$. Thus $(Y, \S, E')$ is soft hyperconnected.
	\end{proof}

	\section{Soft Somewhat Open Functions}\label{se5}\
	
	In this section, we formulate the concepts of soft somewhat open functions (briefly soft $sw$-open) and study its main properties. We characterized it using soft closed and soft dense sets.
	\begin{definition}
		Let $(X, \T, E)$ and $(Y, \S, E')$ be soft topological spaces. A function \sfxy is soft $sw$-open if for each soft open set  $U_E$ over $X$, $f(U_E)$ is soft $sw$-open over $Y$.
	\end{definition}
	
	\begin{remark}
		Let $(X, \T, E)$ and $(Y, \S, E')$ be soft topological spaces. A function \sfxy is soft $sw$-open iff for each non-null soft open set  $U_E$ over $X$, there exits a non-null soft $sw$-open set $V_{E'}$ over $Y$ such that $V_{E'}\sqsubseteq f(U_E)$.	
	\end{remark}
	
	For a single soft point, we have
	\begin{proposition}
		Let $(X, \T, E)$ and $(Y, \S, E')$ be soft topological spaces. A function \sfxy is soft $sw$-open at $P^x_e\in X_E$ if for each soft open set  $U_E$ over $X$ containing $P^x_e$, there exits a soft $sw$-open set $V_{E'}$ over $Y$ such that $f(P^x_e)\in V_{E'}\sqsubseteq f(U_E)$.
	\end{proposition}

	From \cite[Proposition 4.7]{shami2020sd-cont}, Lemma \ref{semi=int not 0} and Remark \ref{w-relation sets}, one can obtain the following for functions:
	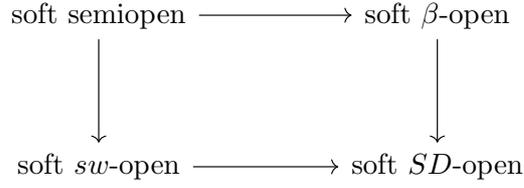
\begin{figure}[H]
		\centering
		\begin{tikzcd}[arrows=rightarrow]
			&\text{soft semiopen}\arrow[dd]\arrow[rr]&& \text{soft $\beta$-open}\arrow[dd]\\
			&&&\\
			&\text{soft $sw$-open}\arrow[rr]&&\text{soft $SD$-open}
		\end{tikzcd}
		\caption*{\quad Diagram III: Relationship between some generalizations of soft openness}\label{D3}
	\end{figure}
	
	None of the implications in the above diagram is reversible and counterexamples are not difficult to obtain.


	\begin{proposition}
		Let $(X, \T, E)$ and $(Y, \S, E')$ be soft topological spaces. For a function \sfxy, the following  are equivalent:
		\begin{enumerate}[(1)]
			\item $f$ is soft $sw$-open,
			\item $f(\Int(G_E))\sqsubseteq \Int_{sw}(f(G_E))$, for each set  $G_E$ over $X$,
			\item $f^{-1}(\Cl_{sw}(H_{E'}))\sqsubseteq \Cl(f^{-1}(H_{E'}))$, for each set  $H_{E'}$ over $Y$.
		\end{enumerate}
	\end{proposition}
	\begin{proof}
		Standard.
	\end{proof}

	\begin{theorem}
		Let $(X, \T, E)$ and $(Y, \S, E')$ be soft topological spaces and let $G_E$ be a soft open subspace over $X$. If \sfxy is soft $sw$-open over $X$, then $f|_{G_E}$ is $sw$-open over $G$.
	\end{theorem}
	\begin{proof}
		If $U_E$ is any soft open over $G_E$, then $U_E$ is also soft open over $X$ because $G_E$ is soft open. By assumption, $f(U_E)$ is soft $sw$-open and hence the result.
	\end{proof}

	\begin{theorem}
		Let $(X, \T, E)$ and $(Y, \S, E')$ be soft topological spaces and let $D_E$ be a soft dense subspace over $X$. If $f:(D, \T_D, E)\to (Y, \S, E')$ is a soft $sw$-open function, then each extension of $f$ is soft $sw$-open over $X$.
	\end{theorem}
	\begin{proof}
		Let $g$ be any extension of $f$ and let $U_{E}$ be a soft open set over $X$. Since $D_E$ is soft dense over $X$, so $U_E\bigsqcap D_E$ is a non-null soft open set over $D_E$. By assumption, there exists a non-null soft $sw$-open set $V_{E'}$ over $Y$ such that $V_{E'}\sqsubseteq f(U_E\bigsqcap D_E)=g(U_E\bigsqcap D_E)\sqsubseteq g(U_E)$. Thus $g$ is soft $sw$-open over $X$.
	\end{proof}

	\begin{theorem}
		Let $(X, \T, E)$ and $(Y, \S, E')$ be soft topological spaces. Let \sfxy be a function and $\{G_E^\alpha: \alpha\in\Lambda\}$ be any soft cover over $X$. Then $f$ is soft $sw$-open, if $f|_{G_E^\alpha}$ is soft $sw$-open for each $\alpha\in\Lambda$.
	\end{theorem}
	\begin{proof}
		Let $U_{E}$ be a (non-null) soft open set over $X$. Then $U_E\bigsqcap G_E^\alpha$ is a non-null soft open set in $G_E^\alpha$ for each $\alpha$. By assumption, $f\left(U_E\bigsqcap G_E^\alpha\right)$ is a soft $sw$-open set over $Y$. But $$f(U_{E})=\bigsqcup f\left(U_E\bigsqcap G_E^\alpha\right),$$ which a union of soft $sw$-open sets and by Lemma \ref{union-sw}, $f(U_{E})$ is a soft $sw$-open set over $Y$. Hence $f$ is soft $sw$-open.
	\end{proof}

	\begin{theorem}\label{characterization1 sw-open}
		Let $(X, \T, E)$ and $(Y, \S, E')$ be soft topological spaces. For a one to one function $f$ from $(X, \T, E)$ onto $(Y, \S, E')$, the following  are equivalent:
		\begin{enumerate}[(1)]
			\item $f$ is soft $sw$-open,
			\item for each soft closed set  $F_{E}$ over $X$ with $f(F_{E})\neq Y_{E'}$, there exists a proper soft closed $K_{E'}$ over $Y$ such that $f(F_{E})\sqsubseteq K_{E'}$.
		\end{enumerate}
	\end{theorem}
	\begin{proof}
		(1)$\implies$(2) Let $F_E$ be a soft closed over $X$ with $f(F_{E})\neq Y_{E'}$. This implies $X_E\setminus F_E$ is a non-null soft open set over $X$. By (1), there exists a soft open set $H_{E'}$ over $Y$ such that $\Phi_{E'}\ne H_{E'}\sqsubseteq f\left(X_E\setminus F_E\right)$. Therefore $f(F_E)=Y_{E'}\setminus(f(X_E\setminus F_E))\sqsubseteq Y_{E'}\setminus H_{E'}$. Set $K_{E'}=Y_{E'}\setminus H_{E'}$. So $K_{E'}$ is a soft closed set over $Y$ that satisfies the required property.
		
		(2)$\implies$(1) Reverse the above steps.
	\end{proof}

	\begin{theorem}\label{characterization2 sw-open}
		Let $(X, \T, E)$ and $(Y, \S, E')$ be soft topological spaces. For a function $f$ from $(X, \T, E)$ onto $(Y, \S, E')$, the following  are equivalent:
		\begin{enumerate}[(1)]
			\item $f$ is soft $sw$-open,
			\item for each soft dense set  $D_{E'}$ over $Y$, then $f^{-1}(D_{E'})$ is soft dense over $X$.
		\end{enumerate}
	\end{theorem}
	\begin{proof}
		(1)$\implies$(2) Let $D_{E'}$ be a soft dense set over $Y$. Suppose otherwise that $f^{-1}(D_{E'})$ is not soft dense over $X$. Then there is a soft closed $K_E$ over $X$ such that $f^{-1}(D_{E'})\sqsubset K_E\ne X_E$. But $X_E\setminus K_E$ is soft open over $X$ so, by (1), there exists a soft open set $V_{E'}$ over $Y$ such that $\Phi_{E'}\ne V_{E'}\sqsubseteq f(X_E\setminus K_E)$. Therefore $V_{E'}\sqsubseteq f(X_E\setminus K_E)\sqsubseteq f(f^{-1}(Y_{E'}\setminus D_{E'}))\sqsubseteq Y_{E'}\setminus D_{E'}$. Thus $D_{E'}\sqsubseteq Y_{E'}\setminus V_{E'}\ne\Phi_{E'}$. But $Y_{E'}\setminus V_{E'}$ is soft closed over $Y$ which violates the soft density of $D_{E'}$ over $Y$. Hence $f^{-1}(D_{E'})$ must be soft dense over $X$.
		
		(2)$\implies$(1) W.l.o.g, let $U_E$ be a non-null soft open set over $X$. We need to prove that $\Int_Y(f(U_E))\ne\Phi_{E'}$. Assume $\Int_Y(f(U_E))=\Phi_{E'}$. Then $\Cl_Y(Y_{E'}\setminus f(U_E))=Y_{E'}$. By (2), $\Cl_X\left(f^{-1}\left(Y_{E'}\setminus f(U_E)\right)\right)=X_E$. But $f^{-1}\left(Y_{E'}\setminus f(U_E)\right)\sqsubseteq X_E\setminus U_E$ and $X_E\setminus U_E$ is soft closed over $X$. Therefore $X_E=\Cl_X(f^{-1}\left(Y\setminus f(U_E)\right))\sqsubseteq X_E\setminus U_E$. This means that $U_E=\Phi_E$, which is contradiction. Thus $\Int_Y(f(U_E))\ne\Phi_{E'}$ and hence $f$ is soft $sw$-open.
	\end{proof}
	
	In the rest of this section, we define an $sw$-homeomorphism and show some soft topological properties which do not keep by soft $sw$-homeomorphisms

	A soft one to one function $f$ from $(X, \T, E)$ onto $(Y, \S, E')$ is called $sw$-homeomorphism if it is soft $sw$-continuous and soft $sw$-open. One can easily conclude that each homeomorphism is $sw$-homeomorphism but not the converse. Evidently, if $f$ is soft $sw$-homeomorphism from $(X, \T, E)$ onto $(Y, \S, E')$, $f^{-1}$ is $sw$-open.
	
	It is worth stating that soft $sw$-homeomorphism does not preserve interesting soft topological properties, as showing in the following examples.
	
	\begin{example}\label{ex2}
		Let $X=Y=\mathbb{R}$ be the set of real numbers and let $E=\{e\}$ be a set of parameters. If $\T$ is the soft topology on $X$ generated by  $\{(e, B(e)):B(e)=(a, b); a, b\in\mathbb{R}; a<b\}$ and  $\S$ is the soft topology on $Y$ generated by  $\{(e, B(e)):B(e)=[a, b); a, b\in\mathbb{R}; a<b\}$ (called soft Sorgenfrey line), then the identity function $i:(X, \T, E)\to(Y, \S, E)$ is soft $sw$-homeomorphism and $(X, \T, E)$ is soft metrizable, soft locally compact and soft connected, while $(Y, \S, E)$ does not have any of these properties.
		
		If we take $A=[0,1]$, then $i|_{A_E}$ is soft $sw$-homeomorphism and $(A, \T_A, E)$ is soft compact, but $(A, \S_A, E)$ is not.
	\end{example}

	\begin{example}\label{ex3}
		Consider $X, E$ and $\T$ given in Example \ref{ex2}. Let $\sigma=\{\Phi_E, X_E, \T\setminus\{G_E:G_E\in\T, (e,{0})\text{ or }(e,{1})\in G_E \}\}$ be another soft topology over $X$. The identity function $i:(X, \T, E)\to(X, \sigma, E)$ is soft $sw$-homeomorphism and $(X, \T, E)$ is soft Hausdorff but $(X, \sigma, E)$ is not soft $T_0$ (consequently, not soft $T_1$).
	\end{example}

	\section{Conclusion and future works}\label{se7}\
	
	Uncertain phenomena exist in many aspects of our daily life. One of the theories proposed to handle uncertainty is the soft set theory. Typologists applied soft sets to initiate a new mathematical structure called soft topology which is the framework of this study.
	
	In this article, we have introduced the concept of soft somewhat open sets as a new generalization of soft open sets. We have shown that the family of soft somewhat open sets lies between the families of soft semiopen sets and soft somewhere dense sets on one hand. On the other hand, the families of soft somewhat open sets and soft $\beta$-open sets are independent of each other. These relationships have been illustrated and main properties have been established with the aid of examples.  Then, we have employed soft somewhat open sets to define soft somewhat continuous, and soft somewhat open functions. We have characterized these two functions and investigated the main features. Some nice connections under certain soft topological space are studied in \cite{zanyar}. The reason for defining these concepts was to discuss the differences between soft homeomorphism and soft somewhat homeomorphism regarding the preservation of certain soft topological properties.
	
	In the upcoming work, we plan to study some topological concepts using soft somewhat open sets such as soft compactness, soft Lindel\"{o}fness, and soft connectedness. The investigation of some applications soft somewhat homeomorphisms is also planned
	
	Furthermore, we explore soft somewhat open sets in the content of supra soft topology.
	
	\bigskip
	
	\textbf{Acknowledgements.} We would like to thank the three anonymous referees for valuable comments that improved the quality of the paper.
	
	%
	

\begin{thebibliography}{9}
		\bibitem{akdag2014alpha} Akdag M. and Ozkan A., (2014), Soft $\alpha$-open sets and soft $\alpha$-continuous functions, Abstr. Appl. Anal., 2014 , pp. 1--7.
		
		\bibitem{ali2009onsome}
		Ali M., Feng F., Liu X., Min W. K. and Shabir M., (2009), On some new operations in soft set theory, Comput Math Appl, 57 , pp. 1547--1553.
		
		\bibitem{ramz}
		Allam A., Ismail T. and Muhammed R., (2017), A new approach to soft belonging, Ann. Fuzzy Math. Inform, 13, pp. 145--152.
		
		\bibitem{shami2018somewhere}
		Al-shami T. M., (2018), Soft somewhere dense sets on soft topological spaces, Commun. Korean Math. Soc, 33(4), pp. 1341--1356.
		
		\bibitem{shami2020sd-cont}
		Al-shami T. M., Alshammari I. and Asaad B. A., (2020), Soft maps via soft somewhere dense sets, Filomat, 34(10), pp. 3429--3440.
		
		
		\bibitem{zanyar}
		Al-shami T. M., Ameen Z. A. and Asaad B. A., Soft bi-continuity and related soft functions, to appear.
		
		\bibitem{Al-shami}
		Al-shami T. M. and El-Shafei M. E., (2020), $T$-soft equality relation, Turkish Journal of Mathematics, 44(94), pp. 1427-1441.
		
		\bibitem{compact}
		Ayg\"{u}no\v{g}lu A. and Ayg\"{u}n H., (2012), Some notes on soft topological spaces, Neural Comput \& Applic, 21, pp. 113–119.
		
		\bibitem{bayramov}
		Bayramov S. and Gunduz C., (2018), A new approach to separability and compactness in soft topological spaces, TWMS Journal of Pure and Applied Mathematics, 9(21), pp. 82--93.
		
		\bibitem{cagman2011soft}
		\c{C}a\u{g}man N., Karata\c{s} S. and  Enginoglu S., (2011), Soft topology, Comput Math Appl, 62, pp. 351–358.
		
		\bibitem{chen2013semi}
		Chen B., (2013), Soft semi-open sets and related properties in soft topological spaces, Appl. Math. Inf. Sci., 7, pp. 287-294.
		
		\bibitem{metric}
		Das S. and Samanta S., (2013), Soft metric, Annals of Fuzzy Mathematics and Information, 6(1), pp. 77-94.
		
		\bibitem{El-shafei2018}
		El-ShafeiM. E., Abo-Elhamayel M. and Al-shami T. M., (2018), Partial soft separation axioms and soft compact spaces, Filomat, 32(13,), pp. 4755-4771.
		
		\bibitem{El-shafei2021}
		El-Shafei M. E. and  Al-shami T. M., (2021), Some operators of a soft set and soft connected spaces using soft somewhere dense sets, Journal of Interdisciplinary Mathematics, , Accepted.
		
		\bibitem{hussein2011some}
		Hussain S. and  Ahmad B., (2011), Some properties of soft topological spaces, Comput Math Appl, 62, pp. 4058-4067.
		
		\bibitem{kandil2014hyperconnected}
		Kandil A., Tantawy O., El-Sheikh S. and  Abd El-latif A., (2014), Soft connectedness via soft ideals, J. New Results in Science, 4, pp. 90-108.
		
		\bibitem{kharal2011softmapping}
		Kharal A. and Ahmad B., (2011), Mappings of soft classes, New Math. Nat. Comput., 7(3), pp. 471-481.
		
		\bibitem{connected}
		Lin F., (2013), Soft connected spaces and soft paracompact spaces, Int. J. Eng. Math., 7(2), pp. 1-7.
		
		\bibitem{softsemicont}
		Mahanta J. and  Das P. K., (2014), On soft topological space via semiopen and semiclosed soft sets, Kyungpook Math J., 4, pp. 221-23.
		
		\bibitem{maji2003onsoft}
		Maji P. K., Biswas R. and Roy A. R., (2003), Soft set theory, Computers and Mathematics with Applications, 45, pp. 555–562.
		
		\bibitem{mold1999}
		Molodtsov D., (1999), Soft set theory first results, Comput Math Appl, 37, pp. 19-31.
		
		\bibitem{softcont}
		Nazmul S. K. and  Samanta S. K., (2013), Neighbourhood properties of soft topological spaces, Ann. Fuzzy Math. Inform., 6, pp. 1-15.
		
		\bibitem{rong-separable}
		Rong W., (2012), The countabilities of soft topological spaces, International Journal of Computational and Mathematical Sciences, 6, pp. 159-162.
		
		\bibitem{shabir2011onsoft}
		Shabir M. and Naz M.,  (2011), On soft topological spaces, Comput Math Appl, 61, pp. 1786–1799.
		
		\bibitem{terepeta}
		Terepeta M., (2019), On separating axioms and similarity of soft topological spaces, Soft Comput, 23, pp. 1049–1057.
		
		\bibitem{yumak2015beta}
		Yumak Y. and Kaymakci  A. K., (2015), Soft $\beta$-open sets and their applications, J. New Theory, 4, pp. 80-89.
	\end{thebibliography}
\end{document}